\documentclass{article}

\usepackage{subfiles}
\usepackage{packages/commonpackage}
\usepackage{packages/package_for_list_of_content}
\usepackage{packages/package_to_delete}
\usepackage{packages/macro}

\numberwithin{equation}{section}
\theoremstyle{plain}
\newtheorem{theorem}{Theorem}[section]
\newtheorem{proposition}[theorem]{Proposition}

\newtheorem{lemma}[theorem]{Lemma}

\newtheorem{remark}[theorem]{Remark}

\newcommand{\Mst}{M_{\mathrm{st}}}
\newcommand{\Mloc}{M_{\mathrm{loc}}}
\newcommand{\Mrec}{M_{\mathrm{rec}}}
\newcommand{\ind}{\mathbf 1}
\newcommand{\Id}{\mathrm{Id}}
\newcommand{\cK}{\mathcal K}

\title{Sharp $L^p$ estimates for the strong spherical maximal operator}
\author{Mingfeng Chen}
\date{}

\begin{document}
\maketitle

\begin{abstract}
    We prove that the strong spherical maximal operator in $\R^n$
    is bounded on $L^p$ for $n\geq4$ and $p>(n+1)/(n-1)$.
    This establishes the conjecture of Hickman and Zahl in the
    remaining dimensions. The main ingredient
    is a local $L^2$ estimate for the pieces obtained by a dyadic
    decomposition in the normal variable. After squaring one input
    coordinate, we reduce this estimate to a $TT^*$ argument for
    diagonal quadratic phases.
\end{abstract}

\section{Introduction}
\subsection{Statement of the main theorem}

Let $\sigma$ be the normalized surface measure on $S^{n-1}$. For $r=(r_1,\ldots,r_n)\in(0,\infty)^n$, write
\begin{align}
    \delta_ru=(r_1u_1,\ldots,r_nu_n),
    \qquad
    \int f\,d\sigma_{r}=\int_{S^{n-1}}f(\delta_r\omega)\,d\sigma(\omega).
\end{align}
The strong spherical maximal operator is defined by
\begin{align}\label{eq:strong}
    \Mst f(x)=\sup_{r\in(0,\infty)^n}
    \left|\int_{S^{n-1}}f(x-\delta_r\omega)\,d\sigma(\omega)\right|.
\end{align}
The operator $\Mst$ is a multiparameter generalization of Stein's
spherical maximal function \cite{Stein}. We follow the notation of
Hickman and Zahl \cite{HZ}. In particular, $\sigma_r$ is the pushforward
of $\sigma$ under $\delta_r$, which in general differs from normalized
Euclidean surface measure on the ellipsoid. We write
\begin{align}\label{eq:ellipsoid}
    E(x;r)=\{y\in\R^n:F_{x,r}(y)=0\},\qquad
    F_{x,r}(y)=\sum_{i=1}^n\frac{(y_i-x_i)^2}{r_i^2}-1,
\end{align}
and define the local operator by
\begin{align}\label{eq:local-maximal}
    \Mloc f(x)=\sup_{r\in[1,2]^n}|f*\sigma_r(x)|.
\end{align}

Our main result is the following.
\begin{theorem}\label{thm:main}
    Let $n\geq4$ and
    \begin{align}\label{eq:range}
        \frac{n+1}{n-1}<p\leq\infty.
    \end{align}
    Then
    \begin{align}\label{eq:main}
        \|\Mst f\|_{L^p(\R^n)}\lesim_{n,p}\|f\|_{L^p(\R^n)}
    \end{align}
    for every $f\in C_c(\R^n)$.
\end{theorem}
\begin{remark}
After completing this manuscript, we learned that J.~B.~Lee,
S.~Lee, and J.~Oh \cite{LLOstrongellr} had independently obtained
the same sharp range in the Euclidean case, as part of a more general
result for strong maximal operators associated with $\ell^r$ spheres.
More precisely, for
\[
    S_r=\left\{y\in\mathbb R^n:
    \sum_{j=1}^n |y_j|^r=1\right\},
    \qquad r\geq2,
\]
they prove that the corresponding strong maximal operator is bounded
on $L^p(\mathbb R^n)$ if and only if
\[
    p>
    \max\left\{
        \frac{n+1}{n-1},
        \frac{r}{n-1}
    \right\}.
\]
In particular, when $r=2$ and $n\geq4$, their result gives precisely
the range in Theorem~\ref{thm:main}. The two proofs are independent
and use different methods.
\end{remark}

For $n\geq3$, Lee, Lee and Oh \cite{LLO} proved \eqref{eq:main}
in the range $p>2(n+1)/(n-1)$. Hickman and Zahl \cite{HZ} improved
this to $p>2$ and conjectured the range \eqref{eq:range}.
They also showed that this range is necessary and that the estimate
fails at the endpoint; see \cite[Appendix~A]{HZ}.
Theorem~\ref{thm:main} proves their conjecture for $n\geq4$.

The case $n=3$ was proved in \cite{HZ}. For $n=2$, the bound follows
from Chen, Guo and Yang \cite{CGY}, or by combining Lee, Lee and
Oh \cite{LLOell} with Pramanik, Yang and Zahl \cite{PYZ}.
Together with these results, Theorem~\ref{thm:main} gives
\eqref{eq:main} for every $n\geq2$ and $p>(n+1)/(n-1)$.

\subsection{Idea of the proof}

We briefly explain the main idea of the proof. We first restrict the radii to $[1,2]^n$ and decompose the surface measure into
oscillatory pieces. The main step is to prove an $L^2$ estimate
with decay for each piece. We then extend this estimate to all
positive radii and interpolate with an elementary $L^q$ bound.

Fix $\chi\in C_c^\infty(\{u:1/2<|u|<2\})$ and
$\beta\in C_c^\infty(\{\tau:1/2\leq|\tau|\leq2\})$. For $\lambda\geq1$, define
\begin{align}\label{eq:normal-kernel}
    K_\lambda(u)
    =\chi(u)\lambda\check\beta\bigl(\lambda(|u|^2-1)\bigr).
\end{align}
For $r\in(0,\infty)^n$, let $K_{\lambda,r}$ be the dilation of
$K_\lambda$. Explicitly,
\begin{align*}
    K_{\lambda,r}(y)
    &=\frac{1}{r_1\cdots r_n}K_\lambda(\delta_r^{-1}y)\\
    &=\frac{\lambda}{r_1\cdots r_n}
    \chi\left(\frac{y_1}{r_1},\ldots,\frac{y_n}{r_n}\right)
    \check\beta\left(\lambda\left(\sum_{i=1}^n
    \frac{y_i^2}{r_i^2}-1\right)\right).
\end{align*}
Here $\lambda$ is the frequency dual to the defining function
$|u|^2-1=F_{0,\mathbf1}(u)$. We refer to $\lambda$ as the normal
frequency.
Fourier inversion in this variable decomposes the surface measure
into kernels of the form \eqref{eq:normal-kernel}; the precise
decomposition is given in Section~\ref{sec:main-proof}.
The following local $L^2$ bound is the main estimate of the paper.
Its proof rests on the product kernel estimate
\eqref{eq:intro-product} below.

\begin{proposition}\label{prop:local}
    For $n\geq2$ and $\lambda\geq1$,
    \begin{align}\label{eq:local-bound}
        \left\|\sup_{r\in[1,2]^n}|K_{\lambda,r}*f|\right\|_2
        \lesim \lambda^{-(n-3)/4}\|f\|_2.
    \end{align}
    The implicit constant depends on $n$, $\chi$ and $\beta$.
\end{proposition}

Let us explain where the exponent in \eqref{eq:local-bound} comes from. We divide the kernel into finitely many pieces and consider one on which the last coordinate is negative and bounded away from zero. We continue to write $\chi$ for the cutoff on this piece. To handle the supremum, choose $r(x)\in[1,2]^n$ measurably
as a function of $x=(x',h)$. It suffices to prove an estimate independent of this choice.

Fix $h$ and make the change of variables
\begin{align}\label{eq:intro-change}
    y'=u,\qquad y_n=h+\sqrt{s}.
\end{align}
Then the equation $F_{x,r}(y)=0$ becomes
\begin{align}\label{eq:intro-graph}
    s=\Phi_{x',h}(u)
    :=r_n^2-\sum_{i=1}^{n-1}\frac{r_n^2}{r_i^2}(u_i-x_i)^2.
\end{align}
The point of this change of variables is that $\Phi_{x',h}$ is a sum of quadratic polynomials, each depending on a single coordinate of $u$. This allows us to estimate the corresponding oscillatory integral one coordinate at a time. All derivatives fall on the input variables $(u,s)$, so the measurable dependence of the radii on $x'$ causes no difficulty.

For fixed $h$ and a fixed measurable choice of radii, the operator
in these coordinates acts on functions $g(u,s)$ by
\begin{align*}
    T_{\lambda,h}g(x')
    =\int_{\R^{n-1}}\int_0^\infty
    K_{\lambda,r(x',h)}(x'-u,-\sqrt{s})
    g(u,s)\,\frac{ds\,du}{2\sqrt{s}}.
\end{align*}
The patch condition restricts $s$ to a fixed compact interval
away from zero. We regard $T_{\lambda,h}$ as an operator from
$L^2(du\,ds)$ to $L^2(dx')$. Let $\cK(x',z')$ denote the integral
kernel of $T_{\lambda,h}T_{\lambda,h}^*$. Explicitly,
\begin{align*}
    \cK(x',z')
    &=\int_{\R^{n-1}}\int_0^\infty
    K_{\lambda,r(x',h)}(x'-u,-\sqrt{s})\\
    &\qquad\times
    \overline{K_{\lambda,r(z',h)}(z'-u,-\sqrt{s})}
    \,\frac{ds\,du}{4s}.
\end{align*}
We suppress its dependence on $\lambda$, $h$ and the chosen radii.
The main technical step is to prove
\begin{align}\label{eq:intro-product}
    |\cK(x',z')|
    \lesim
    \lambda\prod_{i=1}^{n-1}(1+\lambda|x_i-z_i|)^{-1/2}
    \ind_{\{|x'-z'|_\infty\leq C\}}.
\end{align}
Here $\cK$ is a kernel on pairs of output points; its definition
is distinct from that of the convolution kernel $K_{\lambda,r}$.
To obtain each factor in this bound, we compare the two quadratic
coefficients in the corresponding coordinate. If their difference
is large, we apply the second derivative form of van der Corput's
lemma. Otherwise, the separation $x_i-z_i$ gives a lower bound
for the first derivative. In either case, we obtain the factor
$(1+\lambda|x_i-z_i|)^{-1/2}$.

The integral of each factor over a bounded interval is
$O(\lambda^{-1/2})$. Schur's test therefore gives a bound
$O(\lambda^{1-(n-1)/2})$ for the $TT^*$ norm. Taking a square root
and integrating in $h$ gives \eqref{eq:local-bound}.
The details are given in Sections~\ref{sec:quadratic}
and~\ref{sec:local}.

To pass to arbitrary radii, we prove the following lemma.
\begin{lemma}\label{lem:global}
    Let $K\in C_c^1(\R^d)$ be supported in $[-R,R]^d$, where $R\geq1$. For $r\in(0,\infty)^d$, write
\begin{align*}
    K_r(y)=\frac{1}{r_1\cdots r_d}
    K\left(\frac{y_1}{r_1},\ldots,\frac{y_d}{r_d}\right).
\end{align*}
    Suppose
    \begin{align}\label{eq:global-hyp}
        \left\|\sup_{t\in[1,2]^d}|K_{t}*f|\right\|_2
        \leq A\|f\|_2,
        \qquad B=\|K\|_{C^1}.
    \end{align}
   Then, for every integer $M\geq3$, we have
    \begin{align}\label{eq:global-conclusion}
        \left\|\sup_{r\in(0,\infty)^d}|K_{r}*f|\right\|_2
        \lesim_{d,R}
        \bigl[(1+M)^{d/2}A+2^{-M}B\bigr]\|f\|_2.
    \end{align}
\end{lemma}

To prove this lemma, we write the radii as $r=2^kt$ with
$t\in[1,2]^d$. For each $k$, we divide the Fourier transform of
the input into low, middle and high frequencies in each coordinate.
The middle frequencies overlap at most $O((1+M)^d)$ times as $k$
varies. Plancherel's theorem and the assumed local bound then
give the term $(1+M)^{d/2}A$.

At low frequencies in one coordinate, we approximate $K$ by its
integral in that coordinate. The resulting kernel satisfies the
same local estimate in a lower dimension, so we can proceed by
induction on $d$. At high frequencies, smoothing $K$ at relative
scale $2^{-M}$ leaves an error controlled by $2^{-M}B$.
We give the proof in Section~\ref{sec:global}.

For \eqref{eq:normal-kernel}, $B\lesim\lambda^2$. Choose $M$ to be
a sufficiently large multiple of $\log(2+\lambda)$. Then
\begin{align}\label{eq:intro-global}
    \left\|\sup_{r>0}|K_{\lambda,r}*f|\right\|_2
    \lesim
    \bigl(\log(2+\lambda)\bigr)^{n/2}
    \lambda^{-(n-3)/4}\|f\|_2.
\end{align}
On the other hand, the pointwise estimate
$|K_\lambda|\lesim\lambda\ind_{[-2,2]^n}$ gives an $L^q$ bound
$O_q(\lambda)$ for every $q>1$. When
$(n+1)/(n-1)<p<2$, interpolation with $q$ sufficiently close to
one gives a negative power of $\lambda$, and hence a summable
bound over the dyadic pieces. For finite $p>2$, we interpolate
\eqref{eq:intro-global} with the uniform $L^\infty$ bound.
This proves Theorem~\ref{thm:main}; see
Section~\ref{sec:main-proof} for the details.

\subsection{Comparison with Hickman and Zahl and with Stein}\label{sec:comparison}

The main new ingredient is the oscillatory $L^2$ estimate
\eqref{eq:local-bound}. We explain its relation to the argument
of Hickman and Zahl. They consider the ellipsoidal annuli and
the associated maximal operator
\begin{align}\label{eq:HZ-annuli}
    E^\delta(x;r)&=\{y\in\R^n:|F_{x,r}(y)|<\delta\},\\
    M^\delta f(x)&=\sup_{r\in[1,2]^n}
    \frac{1}{|E^\delta(x;r)|}\int_{E^\delta(x;r)}|f(y)|\,dy.
\end{align}
For $n\geq3$, their discretised $L^2$ bound is
\begin{align}\label{eq:HZ-thick}
    \|M^\delta f\|_2\lesim_\epsilon\delta^{-\epsilon}\|f\|_2,
    \qquad 0<\delta<1,\quad\epsilon>0.
\end{align}
Their proof is geometric. More precisely, they estimate the intersections of thin ellipsoidal
annuli. They first restrict the centres to lines and remove
the sets where higher order tangency can occur
\cite[Sections~2--3]{HZ}. They then combine \eqref{eq:HZ-thick}
with the Sobolev estimates of \cite{LLO} and use multiparameter
Littlewood--Paley theory to pass from $\Mloc$ to $\Mst$
\cite[Section~4]{HZ}.

Our proof is Fourier analytic and uses the cancellation in $K_\lambda$ to obtain decay
in $\lambda$. We restrict the centres to $x_n=h$ and set
$s=(y_n-h)^2$ on each patch. The resulting $TT^*$ phase is a sum
of $n-1$ quadratic terms, which gives \eqref{eq:intro-product}
and the decay $\lambda^{-(n-3)/4}$. This estimate holds uniformly
for all measurable choices of radii, including those for which
higher order tangency occurs. Thus the difference between the
two proofs lies in the local $L^2$ estimate: our argument uses
oscillation in the defining function, while theirs estimates
volumes of intersections.

The use of cancellation and interpolation also appears in
Stein's proof of the spherical maximal theorem for $n\geq3$
\cite{Stein}. Stein used Fourier decay and Plancherel's theorem
to estimate a square function, proved a maximal $L^2$ bound for
an analytic family of averaging operators, and applied complex
interpolation. Here we obtain the $L^2$ bound from a quadratic
$TT^*$ argument and interpolate each normal frequency piece.

Notice that the decay in \eqref{eq:local-bound} concerns the
oscillatory pieces of the surface measure. The operator $M^\delta$
averages against positive measures of mass one, so its norm cannot
tend to zero as $\delta\to0$; this follows by testing on the
indicator of a large cube. For $n\geq4$, the decay of the
oscillatory pieces is enough to offset their $O_q(\lambda)$
bound on $L^q$. As $q$ tends to one, the exponent obtained by
interpolation tends to
\begin{align}\label{eq:comparison-exponent}
    \frac{n+1}{2p}-\frac{n-1}{2},
\end{align}
which is negative exactly when $p>(n+1)/(n-1)$. This is how we
obtain the range below $p=2$.

The passage to arbitrary radii in Lemma~\ref{lem:global} introduces
only the factor $(\log(2+\lambda))^{n/2}$. Since the power of
$\lambda$ is unchanged, we obtain the same range of $p$ for
the full operator.

\subsection{AI usage}
The use of LLM models was central to the current work. The main proof was generated by GPT-6 through several rounds of discussion with the author. In particular, the author asked whether a $L^2$ decay estimates like Proposition \ref{prop:local} is possible. The author worked through the argument, clarified the steps, and organized and wrote the present manuscript. GPT-6 was also
used for editing and proofreading throughout the paper. The author takes responsibility for the mathematical claims and the final version of the paper.

\section{Preliminaries}\label{sec:preliminaries}

We collect the basic facts used in the proof. The oscillatory
estimates and the $TT^*$ argument will be used in
Section~\ref{sec:quadratic}. The maximal function and Fourier
estimates will be used to pass from bounded radii to arbitrary
radii in Section~\ref{sec:global}.

\subsection{Notation and Fourier transforms}

Throughout the paper, $A\lesim B$ means $A\leq CB$ for a constant
$C$ that may depend on fixed parameters but is independent of
the frequency and dilation scales. Products and quotients of
vectors are coordinatewise. For $k\in\Z^d$, write
$2^k=(2^{k_1},\ldots,2^{k_d})$; the notation $r>0$ means that
every coordinate of $r$ is positive.
We denote the characteristic function of a set $E$ by $\ind_E$.

We use the Fourier transform
\begin{align}
    \widehat f(\xi)=\int_{\R^d}e^{-2\pi i x\cdot\xi}f(x)\,dx.
\end{align}
The inverse Fourier transform is
$\check g(x)=\int_{\R^d}e^{2\pi i x\cdot\xi}g(\xi)\,d\xi$.

We will use the following consequence. Suppose that $P_k$ are a countable family of
Fourier multipliers with symbols $p_k$, where $|p_k|\leq C$,
and at most $N$ symbols are nonzero at each frequency. Then
\begin{align}\label{eq:prelim-overlap}
    \sum_k\|P_kf\|_2^2
    =\int_{\R^d}\sum_k|p_k(\xi)|^2|\widehat f(\xi)|^2\,d\xi
    \leq C^2N\|f\|_2^2.
\end{align}
This is the estimate that allows us to sum the middle frequency
pieces in Section~\ref{sec:global}.

\subsection{Maximal functions and interpolation}

Let $\Mrec$ denote the centred strong rectangular maximal operator,
\begin{align}
    \Mrec f(x)=\sup_{v>0}\frac{1}{\prod_{i=1}^d 2v_i}
    \int_{\{|y_i-x_i|\leq v_i,\ 1\leq i\leq d\}}|f(y)|\,dy.
\end{align}
We have
\begin{align}\label{eq:prelim-rectangular}
    \|\Mrec f\|_p\lesim_{d,p}\|f\|_p,
    \qquad 1<p\leq\infty.
\end{align}
The bounds hold when the supremum is taken in only
some of the coordinates.

We also use a consequence for convolution kernels. If
\begin{align*}
    |H_v(y)|\leq C\prod_{i=1}^d
    v_i^{-1}(1+|y_i|/v_i)^{-3},\qquad v>0,
\end{align*}
then dyadic decomposition in each coordinate gives
\begin{align}\label{eq:prelim-product-max}
    |H_v*f|\lesim_d C\Mrec f.
\end{align}
The constant is independent of $v$, and the same bound holds
for convolution in a subset of the coordinates.

To estimate a supremum of convolutions, it suffices to consider
a measurable choice of kernel at each point. Indeed, for
a finite collection $\{K_{r^{(1)}},\ldots,K_{r^{(N)}}\}$ and a
fixed $f$, choose the first index at which
$|K_{r^{(j)}}*f(x)|$ is largest. This gives a measurable function
$r(x)$ with
\begin{align*}
    \max_{1\leq j\leq N}|K_{r^{(j)}}*f(x)|
    =|K_{r(x)}*f(x)|.
\end{align*}
A bound uniform in all measurable choices of $r(x)$ therefore
implies the same bound for the finite maximum. We then pass to
a countable dense set of radii by monotone convergence.
For the kernels used below, continuity of convolution in the
radii gives the full supremum.

For interpolation, fix the choice $r(x)$ first. The resulting
operator $L$ is linear. If its $L^{p_0}$ and $L^{p_1}$ norms are
bounded by $A_0$ and $A_1$, respectively, where
$1\leq p_0,p_1\leq\infty$, the Riesz--Thorin theorem
gives
\begin{align}\label{eq:prelim-interpolation}
    \|Lf\|_p\leq A_0^{1-\theta}A_1^\theta\|f\|_p,
    \qquad
    \frac1p=\frac{1-\theta}{p_0}+\frac{\theta}{p_1},
    \quad 0<\theta<1.
\end{align}
When the endpoint bounds are uniform in $r(x)$, the interpolated
bound is also uniform. The preceding argument then recovers the
maximal operator.

\subsection{One-dimensional oscillatory integrals}

We use the first and second derivative forms of van der Corput's
lemma. We state them without an amplitude, since the amplitude
in Section~\ref{sec:quadratic} will be handled by integration
by parts.

\begin{lemma}\label{lem:prelim-vdc}
    Let $J$ be a bounded interval, let $\phi\in C^2(J)$ be
    real-valued, and let $\eta>0$.
    If $\phi'$ is monotone and $|\phi'|\geq\eta$ on $J$, then
    \begin{align}\label{eq:prelim-vdc-first}
        \left|\int_J e^{2\pi i\phi(t)}\,dt\right|
        \lesim\eta^{-1}.
    \end{align}
    If $|\phi''|\geq\eta$ on $J$, then
    \begin{align}\label{eq:prelim-vdc-second}
        \left|\int_J e^{2\pi i\phi(t)}\,dt\right|
        \lesim\eta^{-1/2}.
    \end{align}
    Both estimates hold on every subinterval of $J$, with the
    same absolute constants.
\end{lemma}

\begin{proof}
    For the first estimate, write
    $e^{2\pi i\phi}=(2\pi i\phi')^{-1}(e^{2\pi i\phi})'$
    and integrate by parts. The boundary terms are
    $O(\eta^{-1})$, and monotonicity of $\phi'$ bounds the total
    variation of $1/\phi'$ by $2\eta^{-1}$.

    For the second estimate, $\phi''$ has a fixed sign, so $\phi'$
    is monotone. The set where $|\phi'|\leq\sqrt\eta$ is an interval
    of length at most $2\eta^{-1/2}$. Its complement consists of
    at most two intervals, on each of which the first estimate
    applies with $\sqrt\eta$ in place of $\eta$.
\end{proof}

\subsection{The $TT^*$ argument and Schur's test}

Suppose an integral operator is written as
\begin{align*}
    Tg(x)=\int A(x,y)g(y)\,dy.
\end{align*}
Assume initially that $A$ is bounded and compactly supported.
With respect to Lebesgue measure, the kernel of $TT^*$ is
\begin{align}\label{eq:prelim-ttstar}
    \mathcal H(x,z)=\int A(x,y)\overline{A(z,y)}\,dy.
\end{align}
If
\begin{align}\label{eq:prelim-schur-hyp}
    \sup_x\int|\mathcal H(x,z)|\,dz\leq D,
    \qquad
    \sup_z\int|\mathcal H(x,z)|\,dx\leq D,
\end{align}
Schur's test gives
\begin{align}\label{eq:prelim-schur}
    \|T\|_{2\to2}\leq D^{1/2}.
\end{align}
For a detailed proof, see \cite{Folland}.

\subsection{The coarea formula near the sphere}

For $F(u)=|u|^2-1$, the coarea formula gives
\begin{align*}
    \int_{\R^n}a(u)\delta(F(u))\,du
    =\int_{S^{n-1}}\frac{a(\omega)}{|\nabla F(\omega)|}\,dS(\omega)
    =\frac12\int_{S^{n-1}}a(\omega)\,dS(\omega),
\end{align*}
for continuous compactly supported $a$. Here $dS$ is Euclidean
surface measure and $\delta$ is the Dirac distribution.
Thus the normalized surface measure satisfies
\begin{align}\label{eq:prelim-coarea}
    \sigma=\frac{2}{|S^{n-1}|}\delta(|u|^2-1)\,du.
\end{align}
This identity explains both the use of Fourier inversion in
$|u|^2-1$ and the normalization of the kernels in
Section~\ref{sec:main-proof}.

\section{A quadratic oscillatory integral estimate}\label{sec:quadratic}

The local estimate in Section~\ref{sec:local} will follow from
Lemma~\ref{lem:graph} below, with $m=n-1$. We first prove the
oscillatory integral estimate needed to bound its $TT^*$ kernel.

\subsection{A product estimate for quadratic phases}

The difference of two quadratic phases need not have a large second
derivative. When the quadratic terms nearly cancel, we use the
separation of their centres to control the first derivative instead.
The following lemma applies this observation in each coordinate.

\begin{lemma}\label{lem:product}
    Let $m\geq1$, $R\geq1$ and $\lambda\geq1$. Suppose
    \begin{align}
        c_0\lambda\leq|\tau|\leq C_0\lambda,\qquad
        |\sigma|\leq C_0\lambda,\qquad
        b_-\leq b_i,c_i\leq b_+
    \end{align}
    for fixed positive constants $c_0,C_0,b_-,b_+$.
    For $x,z\in\R^m$, put
    \begin{align}
        P(u)=\sum_{i=1}^m
        \bigl[\tau b_i(u_i-x_i)^2-\sigma c_i(u_i-z_i)^2\bigr].
    \end{align}
    If $a\in C_c^m(\R^m)$ is supported in $z+[-R,R]^m$, then
    \begin{align}\label{eq:product}
        \left|\int_{\R^m}e^{2\pi iP(u)}a(u)\,du\right|
        \lesim
        \|\partial_1\cdots\partial_m a\|_{L^1}
        \prod_{i=1}^m(1+\lambda|x_i-z_i|)^{-1/2}.
    \end{align}
    The implicit constant depends only on $m$, $R$ and the fixed
    constants above.
\end{lemma}

\begin{proof}
    Fix a coordinate $i$, and write
    \begin{align*}
        P_i(t)=\tau b_i(t-x_i)^2-\sigma c_i(t-z_i)^2,\qquad
        v_i=x_i-z_i,\qquad A_i=\tau b_i-\sigma c_i.
    \end{align*}
    Thus $P(u)=\sum_iP_i(u_i)$, and
    \begin{align}\label{eq:phase-derivatives}
        P_i'(t)=2A_i(t-z_i)-2\tau b_i v_i,\qquad
        P_i''(t)=2A_i.
    \end{align}
    We will show that every subinterval $J$ of
    $I_i=[z_i-R-1,z_i+R+1]$ satisfies
    \begin{align}\label{eq:primitive}
        \left|\int_J e^{2\pi iP_i(t)}\,dt\right|
        \lesim(1+\lambda|v_i|)^{-1/2}.
    \end{align}
    Using the slightly larger interval $I_i$ will allow us to
    integrate by parts without boundary terms later.

    If $\lambda|v_i|\leq1$, the length of $I_i$ gives
    \eqref{eq:primitive}. Suppose $\lambda|v_i|>1$, and fix
    $0<\kappa<c_0b_-/(2(R+1))$. If
    $|A_i|\geq\kappa\lambda|v_i|$, then
    $|P_i''|\geq2\kappa\lambda|v_i|$, so the second derivative
    estimate \eqref{eq:prelim-vdc-second} gives the required bound.

    If $|A_i|<\kappa\lambda|v_i|$, the first term in $P_i'$ is
    small on $I_i$. The second term therefore dominates, since
    $|\tau b_i|\geq c_0b_-\lambda$. More precisely,
    \begin{align}
        |P_i'(t)|
        &\geq2c_0b_-\lambda|v_i|-2|A_i|(R+1)
        \geq c_0b_-\lambda|v_i|,\quad \forall t\in I_i.
    \end{align}
    The derivative $P_i'$ is affine and hence monotone. The first
    derivative estimate \eqref{eq:prelim-vdc-first} therefore gives
    the stronger bound $O((\lambda|v_i|)^{-1})$. This proves
    \eqref{eq:primitive} in both cases.

    It remains to include the amplitude $a$. For $t\in I_i$, set
    \begin{align}
        G_i(t)=\int_{z_i-R-1}^t e^{2\pi iP_i(v)}\,dv.
    \end{align}
    By \eqref{eq:primitive},
    $\|G_i\|_{L^\infty(I_i)}\lesim(1+\lambda|v_i|)^{-1/2}$.
    Since $G_i'=e^{2\pi iP_i}$, integration by parts once in each
    variable gives
    \begin{align}
        \int_{\R^m}e^{2\pi iP(u)}a(u)\,du
        =(-1)^m\int_{\prod_i I_i}
        \prod_{i=1}^mG_i(u_i)\,
        \partial_1\cdots\partial_m a(u)\,du.
    \end{align}
    All boundary terms vanish because $a$ is supported in the
    interior of $\prod_i I_i$. Taking absolute values and using
    the bounds for $G_i$ proves \eqref{eq:product}.
\end{proof}

\subsection{The \texorpdfstring{$L^2$}{L2} estimate for quadratic graphs}

We now consider an operator whose input variables are
$(u,s)\in\R^m\times\R$ and whose output variable is $x\in\R^m$.
For each $x$, the phase is determined by a quadratic graph
$s=\Phi_x(u)$. Its coefficients may depend measurably on $x$.
The proof differentiates only in $(u,s)$.

\begin{lemma}\label{lem:graph}
    Let $m\geq1$, $R\geq1$ and $\lambda\geq1$. Suppose
    \begin{align}\label{eq:quadratic-graph}
        \Phi_x(u)=b_0(x)-\sum_{i=1}^m b_i(x)(u_i-x_i)^2,
        \qquad 0<b_-\leq b_i(x)\leq b_+,
    \end{align}
    where the coefficients are real and measurable in $x$.
    Let $a_x(u,s)$ be measurable in $x$, supported in
    \begin{align}\label{eq:amplitude-support}
        |u-x|_\infty\leq R,\qquad |s|\leq R,
    \end{align}
    and have uniformly bounded derivatives in $(u,s)$ through
    order $m+4$. Suppose $q_x(\tau)$ is measurable in $(x,\tau)$
    and satisfies
    \begin{align}\label{eq:q-support}
        |q_x(\tau)|\leq C_0
        \ind_{\{c_0\lambda\leq|\tau|\leq C_0\lambda\}}.
    \end{align}
    Define
    \begin{align}\label{eq:graph-operator}
        T_\lambda g(x)
        =\int_{\R^m}\int_\R a_x(u,s)g(u,s)
        \int_\R q_x(\tau)e^{2\pi i\tau(s-\Phi_x(u))}
        \,d\tau\,ds\,du.
    \end{align}
    Then
    \begin{align}\label{eq:graph-bound}
        \|T_\lambda g\|_{L^2(\R^m)}
        \lesim\lambda^{1/2-m/4}\|g\|_{L^2(\R^{m+1})}.
    \end{align}
    The constant depends only on the stated bounds. No derivatives
    of the coefficients with respect to $x$ are needed.
\end{lemma}

\begin{proof}
   We first restrict $x$ to a bounded set. The support conditions on $a_x$ then restrict $(u,s)$ to a bounded set as well. The resulting kernel is bounded and compactly supported. All estimates are independent of the chosen set, and we remove the restriction at the end of the proof.

    Let $\cK(x,z)$ be the kernel of $T_\lambda T_\lambda^*$.
    Multiplying the kernel at $x$ by the conjugate kernel at $z$
    gives the phase
    \begin{align*}
        \tau(s-\Phi_x(u))-\sigma(s-\Phi_z(u))
        =(\tau-\sigma)s+\sigma\Phi_z(u)-\tau\Phi_x(u).
    \end{align*}
    We perform the $s$ integral first, writing
    \begin{align}
        B_{x,z}(u,\eta)
        =\int_\R e^{2\pi i\eta s}
        a_x(u,s)\overline{a_z(u,s)}\,ds.
    \end{align}
    Then
    \begin{align}\label{eq:tt-kernel}
        \cK(x,z)
        =\iint q_x(\tau)\overline{q_z(\sigma)}
        \int e^{2\pi i(\sigma\Phi_z(u)-\tau\Phi_x(u))}
        B_{x,z}(u,\tau-\sigma)\,du\,d\tau\,d\sigma.
    \end{align}

    The $s$ integral gives decay in the difference of the two
    frequencies. For $|\eta|\geq1$, we integrate by parts in
    $s$; for $|\eta|<1$, we use the bounded support and amplitude.
    The same argument after differentiating in $u$ gives
    \begin{align}\label{eq:B-derivative}
        |\partial_u^\alpha B_{x,z}(u,\eta)|
        \lesim(1+|\eta|)^{-2},\qquad |\alpha|\leq m.
    \end{align}
    Also, $B_{x,z}$ is supported where
    \begin{align}
        u\in\bigl(x+[-R,R]^m\bigr)\cap
        \bigl(z+[-R,R]^m\bigr).
    \end{align}
    Thus $\cK(x,z)=0$ when $|x-z|_\infty>2R$.

    We next estimate the $u$ integral in \eqref{eq:tt-kernel}.
    Apart from the constant $\sigma b_0(z)-\tau b_0(x)$, its phase is
    \begin{align}
        \sum_{i=1}^m
        \bigl[\tau b_i(x)(u_i-x_i)^2
        -\sigma b_i(z)(u_i-z_i)^2\bigr].
    \end{align}
    This is the phase in Lemma~\ref{lem:product}, with
    $b_i=b_i(x)$ and $c_i=b_i(z)$. The constant term contributes
    a factor of modulus one. By \eqref{eq:B-derivative} and the
    bounded $u$-support,
    \begin{align*}
        \|\partial_1\cdots\partial_m
        B_{x,z}(\,\cdot\,,\tau-\sigma)\|_{L^1(\R^m)}
        \lesim(1+|\tau-\sigma|)^{-2}.
    \end{align*}
    Lemma~\ref{lem:product} therefore gives
    \begin{align}\label{eq:inner-oscillation}
        \left|\int e^{2\pi i(\sigma\Phi_z(u)-\tau\Phi_x(u))}
        B_{x,z}(u,\tau-\sigma)\,du\right|
        \lesim(1+|\tau-\sigma|)^{-2}
        \prod_{i=1}^m(1+\lambda|x_i-z_i|)^{-1/2}.
    \end{align}

    We can now integrate in the frequencies. The decay in
    $\tau-\sigma$ makes the $\sigma$ integral bounded uniformly
    in $\tau$. Only the $\tau$ integral costs a factor of $\lambda$:
    \begin{align}
        \iint_{\substack{|\tau|\leq C_0\lambda\\
        |\sigma|\leq C_0\lambda}}
        (1+|\tau-\sigma|)^{-2}\,d\sigma\,d\tau
        \leq\int_{|\tau|\leq C_0\lambda}\int_\R
        (1+|\tau-\sigma|)^{-2}\,d\sigma\,d\tau
        \lesim\lambda.
    \end{align}
    Substituting into \eqref{eq:tt-kernel}, we obtain
    \begin{align}\label{eq:tt-product}
        |\cK(x,z)|
        \lesim\lambda
        \prod_{i=1}^m(1+\lambda|x_i-z_i|)^{-1/2}
        \ind_{\{|x-z|_\infty\leq2R\}}.
    \end{align}

    Finally, each factor on the right has integral
    \begin{align}
        \int_{-2R}^{2R}(1+\lambda|t|)^{-1/2}\,dt
        \lesim_R\lambda^{-1/2}.
    \end{align}
    Integrating \eqref{eq:tt-product} in either $x$ or $z$ gives
    $\lambda(\lambda^{-1/2})^m=\lambda^{1-m/2}$, up to a
    constant. Schur's test therefore yields
    \begin{align}
        \|T_\lambda T_\lambda^*\|_{2\to2}
        \lesim\lambda^{1-m/2}.
    \end{align}
    Since $\|T_\lambda\|_{2\to2}^2
    =\|T_\lambda T_\lambda^*\|_{2\to2}$, taking a square root
    proves \eqref{eq:graph-bound}. We remove the restriction on
    $x$ by exhausting $\R^m$ with bounded sets and applying
    monotone convergence.
\end{proof}

\section{The local estimate for ellipsoids}\label{sec:local}

\begin{proof}[Proof of Proposition~\ref{prop:local}]
    We decompose $\chi$ into finitely many smooth patches on each of
    which some coordinate has a fixed sign and absolute value at least
    $c>0$. By permuting and reflecting the coordinates, we may assume
    \begin{align}\label{eq:patch}
        \supp\chi\subset\{u:u_n<-c,\ 1/2<|u|<2\}.
    \end{align}

    It is enough to prove the estimate for a measurable choice of radii
    $r(x)\in[1,2]^n$, with a bound independent of this choice. 
    Write $x=(x',h)$ and fix $h$. We estimate
    \begin{align}\label{eq:linearized-local}
        L_\lambda f(x',h)
        =\int_{\R^n}\frac{\chi(\delta_r^{-1}(x-y))}{r_1\cdots r_n}
        \lambda\check\beta\bigl(\lambda F_{x,r}(y)\bigr)f(y)\,dy,
    \end{align}
    where $r=r(x',h)$.

    On the support of the amplitude, \eqref{eq:patch} gives $y_n>h$,
    so we make the change of variables
    \begin{align}\label{eq:slice-change}
        y'=u,\qquad y_n=h+\sqrt{s},\qquad
        dy_n=\frac{ds}{2\sqrt{s}}.
    \end{align}
    On this support, $|u-x'|_\infty\leq4$ and $c^2\leq s\leq16$.
    The defining function now takes the form
    \begin{align}\label{eq:exact-graph}
        F_{x,r}(u,h+\sqrt{s})
        =\frac{s-\Phi_{x',h}(u)}{r_n^2},
        \qquad
        \Phi_{x',h}(u)=r_n^2-\sum_{i=1}^{n-1}
        \frac{r_n^2}{r_i^2}(u_i-x_i)^2.
    \end{align}
    Fourier inversion gives
    \begin{align}\label{eq:normal-inversion}
        \lambda\check\beta\left(\lambda\frac{s-\Phi}{r_n^2}\right)
        =r_n^2\int_\R e^{2\pi i\tau(s-\Phi)}
        \beta(r_n^2\tau/\lambda)\,d\tau.
    \end{align}
    The operator on the slice $x_n=h$ therefore has the form
    \eqref{eq:graph-operator}, with
    \begin{align}\label{eq:sliced-amplitude}
        a_{x',h}(u,s)
        &=\frac{r_n^2}{2\sqrt{s}\,r_1\cdots r_n}
        \chi\left(\frac{x'-u}{r'},-\frac{\sqrt{s}}{r_n}\right),\\
        q_{x',h}(\tau)&=\beta(r_n^2\tau/\lambda).
    \end{align}
    We extend the amplitude by zero for $s\leq0$. Since its support
    stays away from $s=0$, all input derivatives are uniformly bounded.
    Moreover,
    \begin{align}
        \frac14\leq\frac{r_n^2}{r_i^2}\leq4,\qquad
        \frac{\lambda}{8}\leq|\tau|\leq2\lambda
        \quad\text{on }\supp q_{x',h}.
    \end{align}
    These bounds are independent of the choice of radii, as required
    in Lemma~\ref{lem:graph}. To localize the input to the relevant
    $s$-interval, choose $\zeta\in C_c^\infty((0,\infty))$ equal to one
    on $[c^2,16]$, and set
    \begin{align}
        g_h(u,s)=\zeta(s)f(u,h+\sqrt{s}),\quad s>0,\\
        \qquad g_h(u,s)=0,\quad s\leq0.
    \end{align}
    Lemma~\ref{lem:graph}, with $m=n-1$, yields
    \begin{align}
        \|L_\lambda f(\,\cdot\,,h)\|_{L^2(\R^{n-1})}
        \lesim \lambda^{-(n-3)/4}\|g_h\|_{L^2(\R^n)}.
    \end{align}
    We now integrate over $h$ to recover the full $L^2$ norm.
    By Fubini and a translation in $h$, we have
    \begin{align}\label{eq:slicing-norm}
        \int_\R\|g_h\|_2^2\,dh
        &=\int_{\R^{n-1}}\int_0^\infty\int_\R
        |\zeta(s)|^2|f(u,h+\sqrt{s})|^2\,dh\,ds\,du\\
        &=\|\zeta\|_2^2\|f\|_2^2.
    \end{align}
    This gives the required bound for every measurable choice of radii.
    Passing to the supremum as above and summing over the patches
    proves the proposition.
\end{proof}

\begin{remark}\label{rem:dimension}
    The decay in \eqref{eq:local-bound} holds when $n\geq4$.
    When $n=3$, the estimate is uniform in $\lambda$, but does not
    suffice for the summation in Section~\ref{sec:main-proof}.
\end{remark}

\section{From bounded radii to arbitrary radii}\label{sec:global}

We prove Lemma~\ref{lem:global} by induction on $d$. On each dyadic
rectangle of radii, we divide the frequency in every coordinate into
low, middle and high parts. The middle parts have bounded overlap,
so their contribution follows from the local estimate and Plancherel's
theorem. The high frequencies are controlled by smoothing the kernel.
For the low frequencies, the kernel is approximated by a product
involving a coordinate marginal of $K$, to which we apply the
induction hypothesis. We begin by checking that these marginals
satisfy the same local estimate. Throughout this section, $K$, $A$,
$B$ and $R$ are as in the lemma.

\subsection{Coordinate marginals and product kernels}

For $I\subseteq\{1,\ldots,d\}$ and $J=I^c$, define
\begin{align}\label{eq:marginal}
    K^I(y_I)=\int_{\R^J}K(y_I,y_J)\,dy_J.
\end{align}
For nonempty $I$, $K^I_{r_I}$ denotes the normalized dilation
of $K^I$ in the coordinates indexed by $I$.
When $I=\varnothing$, we interpret $K^I$ as the scalar $\int K$.
To obtain the local estimate for $K^I$, we test
\eqref{eq:global-hyp} on functions that are constant on a large box
in the remaining coordinates. More precisely, let
\begin{align}
    f_L(x)=g(x_I)\ind_{[-L,L]^J}(x_J),
\end{align}
fix $t_J=\mathbf1$, and restrict the output to
$x_J\in[-L+R,L-R]^J$. Since $K$ is supported in $[-R,R]^d$,
the convolution on this smaller box is exactly the convolution
with $K^I$ in the $I$ coordinates. Thus, for $L>R$,
\eqref{eq:global-hyp} gives
\begin{align}\label{eq:marginal-local}
    (2L-2R)^{|J|}
    \left\|\sup_{t_I\in[1,2]^I}
    |K^I_{t_I}*g|\right\|_2^2
    \leq A^2(2L)^{|J|}\|g\|_2^2.
\end{align}
Letting $L\to\infty$ proves the local bound for $K^I$ with the same
constant $A$. When $I=\varnothing$, the same argument gives
\begin{align}\label{eq:mass}
    \left|\int K\right|\leq A.
\end{align}
Also, $K^I$ is supported in $[-R,R]^I$ and
\begin{align}\label{eq:marginal-C1}
    \|K^I\|_{C^1}\lesim_{d,R} B.
\end{align}

The errors in the kernel approximations below will be bounded by
multiples of the product kernels
\begin{align}\label{eq:W}
    W_{v}(y)=\prod_{i=1}^d
    v_i^{-1}(1+|y_i|/v_i)^{-3}.
\end{align}
By \eqref{eq:prelim-product-max}, we have
\begin{align}\label{eq:W-max}
    W_{v}*|f|\lesim_d \Mrec f,
\end{align}
uniformly in $v>0$. We will also use the corresponding bound for
convolution in a subset $J$ of the coordinates. We write $M_J$ for
the rectangular maximal operator in those coordinates and set
$M_\varnothing=\Id$.

\subsection{A frequency decomposition}

Choose a real even function $\vartheta\in C_c^\infty(\R)$ such that
$0\leq\vartheta\leq1$, $\vartheta=1$ on $[-1,1]$, and
$\supp\vartheta\subset[-2,2]$. Write
\begin{align}
    \varphi=\check\vartheta,\qquad
    \varphi_v(t)=v^{-1}\varphi(t/v).
\end{align}
In particular, $\int\varphi=1$. For a coordinate set $I$, write
$\Psi_{v_I}(y_I)=\prod_{i\in I}\varphi_{v_i}(y_i)$.

Write $r=2^kt$, where $k\in\Z^d$ and $t\in[1,2]^d$.
For a fixed integer $M\geq3$, define the frequency cutoffs
\begin{align}\label{eq:frequency-cutoffs}
    \widehat{L_{i,k}f}(\xi)
    &=\vartheta(2^{k_i+M}\xi_i)\widehat f(\xi),\\
    \widehat{H_{i,k}f}(\xi)
    &=\vartheta(2^{k_i-M}\xi_i)\widehat f(\xi).
\end{align}
These satisfy $H_{i,k}L_{i,k}=L_{i,k}$. For the decomposition, put
\begin{align}
    P_k=\prod_{i=1}^d(H_{i,k}-L_{i,k}),\qquad
    L_{J,k}=\prod_{j\in J}L_{j,k},
\end{align}
and
\begin{align}
    E_k=\left(\Id-\prod_{i=1}^d H_{i,k}\right)
    \prod_{i=1}^d(\Id-L_{i,k}).
\end{align}
Since $H_{i,k}(\Id-L_{i,k})=H_{i,k}-L_{i,k}$, expanding the
products gives the exact decomposition
\begin{align}\label{eq:frequency-split}
    \Id=P_k+
    \sum_{\varnothing\ne J\subseteq\{1,\ldots,d\}}
    (-1)^{|J|+1}L_{J,k}+E_k.
\end{align}

The multiplier $P_k$ vanishes unless
\begin{align}
    2^{-M}\leq2^{k_i}|\xi_i|\leq2^{M+1}
    \qquad(1\leq i\leq d).
\end{align}
Thus, for each $\xi$, at most $O_d((1+M)^d)$ of these multipliers
are nonzero. Since they are uniformly bounded,
\eqref{eq:prelim-overlap} gives
\begin{align}\label{eq:middle-square}
    \sum_{k\in\Z^d}\|P_kf\|_2^2
    \lesim_d(1+M)^d\|f\|_2^2.
\end{align}
The local estimate \eqref{eq:global-hyp} applies on every dyadic
rectangle of radii, with the same constant $A$, after rescaling each
coordinate. We can therefore sum the squares of the local estimates
to obtain
\begin{align}\label{eq:middle-bound}
    \left\|\sup_{k\in\Z^d}\sup_{t\in[1,2]^d}
    |K_{2^kt}*P_kf|\right\|_2^2
    &\leq\sum_k
    \left\|\sup_{t\in[1,2]^d}|K_{2^kt}*P_kf|\right\|_2^2\\
    &\leq A^2\sum_k\|P_kf\|_2^2
    \lesim_d A^2(1+M)^d\|f\|_2^2.
\end{align}

\subsection{Two kernel estimates}

Set
\begin{align}
    \epsilon=2^{-M},\qquad
    a_i=2^{k_i-M},\qquad b_i=2^{k_i+M},
    \qquad r=2^kt.
\end{align}
Then
\begin{align}
    \epsilon/2\leq a_i/r_i\leq\epsilon,
    \qquad r_i/b_i\leq2\epsilon.
\end{align}
We now prove the two kernel estimates needed for the high and low
frequency terms. At the small scales $a_i$, smoothing changes the
kernel by an error bounded by
\begin{align}\label{eq:high-kernel}
    |K_{r}-K_{r}*\Psi_{a}|
    \lesim_{d,R}\epsilon B W_{r},
\end{align}
whereas at the large scales $b_j$, smoothing in a nonempty set $J$ of
coordinates replaces $K_r$ by a product involving its marginal
$K^I$, where $I=J^c$. The error satisfies
\begin{align}\label{eq:low-kernel}
    \left|K_{r}*_J\Psi_{b_J}
    -K^I_{r_I}\otimes\Psi_{b_J}\right|
    \lesim_{d,R}\epsilon B W_{\rho},
    \qquad
    \rho_i=
    \begin{cases}
        r_i,&i\in I,\\
        b_i,&i\in J.
    \end{cases}
\end{align}
Here $*_J$ denotes convolution in the coordinates in $J$.

To prove \eqref{eq:high-kernel}, we first rescale by $\delta_r$ and
put $\eta_i=a_i/r_i$. Since $\int\Psi_{\eta}=1$, the mean value
theorem gives
\begin{align}\label{eq:high-near}
    |K(w)-K*\Psi_{\eta}(w)|
    &\leq B\int |z|_1|\Psi_{\eta}(z)|\,dz
    \lesim_d\epsilon B.
\end{align}
On $[-2R-2,2R+2]^d$, the product weight in the rescaled estimate
is bounded below, so this proves the required bound there.
Outside this box, $K(w)=0$, and we use the decay of the smoothing
kernel. Whenever $|w_i|>2R+2$, the Schwartz decay of $\varphi$
gives, for $N\geq4$,
\begin{align}
    \int_{-R}^R
    |\varphi_{\eta_i}(w_i-z_i)|\,dz_i
    \lesim_{R,N}\epsilon^{N-1}(1+|w_i|)^{-N}.
\end{align}
For the coordinates with $|w_i|\leq2R+2$, the integral is at most
$\|\varphi\|_1$, and the corresponding factors in the product weight
are bounded below. Applying the preceding estimate in every remaining
coordinate, at least one of which satisfies $|w_i|>2R+2$, we obtain
\begin{align}
    |K*\Psi_{\eta}(w)|
    \lesim_{d,R}\epsilon B\prod_i(1+|w_i|)^{-3}.
\end{align}
Combining this with \eqref{eq:high-near} and rescaling back proves
\eqref{eq:high-kernel}.

For \eqref{eq:low-kernel}, the definition of the marginal allows us
to write the difference inside the absolute value as
\begin{align}\label{eq:low-difference}
    \int_{\R^J}K_{r}(y_I,z_J)
    \bigl[\Psi_{b_J}(y_J-z_J)-\Psi_{b_J}(y_J)\bigr]\,dz_J.
\end{align}
On the support of $K_{r}$, we have $|z_j|/b_j\leq2R\epsilon$.
Thus the mean value theorem and the Schwartz bounds give
\begin{align}
    |\Psi_{b_J}(y_J-z_J)-\Psi_{b_J}(y_J)|
    \lesim_{d,R}\epsilon
    \prod_{j\in J}b_j^{-1}(1+|y_j|/b_j)^{-3}.
\end{align}
Substituting this into \eqref{eq:low-difference} and integrating in
$z_J$ proves \eqref{eq:low-kernel}, since $|y_i|\leq Rr_i$ for
$i\in I$ on the support of $K_r$.

\subsection{Completion of the induction}

The induction starts with the scalar bound \eqref{eq:mass} in
dimension zero. Assume that the lemma holds in dimensions less than
$d$. The middle frequency term has already been estimated in
\eqref{eq:middle-bound}, so we only need to bound the low and high
frequency terms in \eqref{eq:frequency-split}.

For the low frequencies, fix $J\ne\varnothing$ and let $I=J^c$.
By \eqref{eq:low-kernel} and \eqref{eq:W-max}, we have
\begin{align}\label{eq:low-operator}
    K_{r}*L_{J,k}f
    =\Psi_{b_J}*_J(K^I_{r_I}*_I f)
    +\mathcal E_{k,t,J}f,
    \qquad
    |\mathcal E_{k,t,J}f|
    \lesim_{d,R}\epsilon B\Mrec f.
\end{align}
Define
\begin{align}
    \mathcal G_I f(x)=
    \sup_{r_I>0}|K^I_{r_I}*_I f(x)|.
\end{align}
The first term in \eqref{eq:low-operator} is bounded by
$C_d M_J(\mathcal G_I f)$, uniformly in $k$ and $t$, because
$\Psi_{b_J}$ is a product of Schwartz kernels. When
$I\ne\varnothing$, the estimates \eqref{eq:marginal-local} and
\eqref{eq:marginal-C1} allow us to apply the induction hypothesis
to $K^I$. Applying it for each fixed value of the other coordinates
and then integrating in those coordinates gives
\begin{align}\label{eq:inductive-bound}
    \|\mathcal G_I f\|_2
    \lesim_{d,R}\bigl[(1+M)^{|I|/2}A+\epsilon B\bigr]\|f\|_2.
\end{align}
When $I=\varnothing$, the required bound follows instead from
$\mathcal G_\varnothing f=|\int K|\,|f|$ and \eqref{eq:mass}.
The $L^2$ boundedness of $M_J$, together with the error estimate in
\eqref{eq:low-operator}, now bounds each low frequency term by the
right-hand side of \eqref{eq:global-conclusion}.

It remains to estimate the high frequency term. Set
\begin{align}
    R_{k,t}=K_{2^kt}-K_{2^kt}*\Psi_{a}.
\end{align}
Since $\prod_i H_{i,k}$ is convolution with $\Psi_a$, commutativity
of the convolution operators gives
\begin{align}\label{eq:high-operator}
    K_{2^kt}*E_kf
    =\prod_{i=1}^d(\Id-L_{i,k})(R_{k,t}*f).
\end{align}
By \eqref{eq:high-kernel} and \eqref{eq:W-max},
\begin{align}
    |R_{k,t}*f|\lesim_{d,R}\epsilon B\Mrec f.
\end{align}
We expand the product in \eqref{eq:high-operator} and apply this
bound to each term. Each $L_{J,k}$ is convolution in $J$ with a
product of Schwartz kernels, so
\begin{align}\label{eq:high-bound}
    \sup_{k,t}|K_{2^kt}*E_kf|
    \lesim_{d,R}\epsilon B
    \sum_{J\subseteq\{1,\ldots,d\}}M_J(\Mrec f).
\end{align}
The $L^2$ boundedness of $M_J$ and $\Mrec$ shows that this term
has norm at most $C_{d,R}\epsilon B\|f\|_2$. Combining the high
frequency estimate with \eqref{eq:middle-bound} and the low
frequency estimates proves \eqref{eq:global-conclusion} and
completes the induction.

\section{Proof of the maximal estimate}\label{sec:main-proof}
Assume $n\geq4$. We decompose the surface measure in the normal
variable and apply the estimates proved above to each piece.
Interpolation allows us to sum these estimates.

\subsection{Decomposition in the normal variable}

Choose a radial function
$\chi\in C_c^\infty(\{u:1/2<|u|<2\})$ equal to one near $S^{n-1}$.
Let $\rho\in C_c^\infty(\R)$ be real and even, with
$\rho=1$ on $[-1,1]$ and support in $(-2,2)$. Set
\begin{align}
    \beta(t)=\rho(t)-\rho(2t),\qquad c_n=\frac{2}{|S^{n-1}|},
\end{align}
where $|S^{n-1}|$ denotes unnormalized surface area. Define
\begin{align}\label{eq:dyadic-kernels}
    K_0(u)&=c_n\chi(u)\check\rho(|u|^2-1),\\
    K_j(u)&=c_n\chi(u)2^j\check\beta\bigl(2^j(|u|^2-1)\bigr),
    \qquad j\geq1.
\end{align}
Since
\begin{align}
    \rho(\tau)+\sum_{j=1}^J\beta(2^{-j}\tau)=\rho(2^{-J}\tau),
\end{align}
Fourier inversion gives
\begin{align}\label{eq:partial-kernel}
    K_0(u)+\sum_{j=1}^J K_j(u)
    =c_n\chi(u)2^J\check\rho\bigl(2^J(|u|^2-1)\bigr).
\end{align}
To identify the limit, integrate \eqref{eq:partial-kernel} against
a continuous compactly supported test function and apply the coarea
formula. Since $c_n\delta(|u|^2-1)\,du=\sigma$, we obtain
\begin{align}\label{eq:measure-decomposition}
    \sigma=K_0+\sum_{j\geq1}K_j
\end{align}
both in distributions and against compactly supported continuous
functions.

For $j\geq1$, $K_j$ is $c_n$ times the kernel in
\eqref{eq:normal-kernel} at $\lambda=2^j$, with the cutoffs
chosen above. For $j\geq0$, write $K_{j,r}:=(K_j)_r$. We set
\begin{align}
    M_j f(x)=\sup_{r>0}|K_{j,r}*f(x)|,\qquad
    a=\frac{n-3}{4}>0.
\end{align}
The kernels $K_j$ are supported in a fixed cube, and differentiation
gives
\begin{align}\label{eq:kernel-C1}
    \|K_j\|_{C^1}\lesim_n 2^{2j}.
\end{align}
We also have, for $j\geq1$,
\begin{align}\label{eq:kernel-L1}
    \|K_j\|_1\lesim_n1,
\end{align}
by the coarea formula: the density of the level sets of $|u|^2-1$
is bounded on $\supp\chi$, and $\check\beta$ is integrable.
The same bound for $K_0$ follows directly from its definition.

\subsection{Bounds for a single piece}

Let $j\geq1$. Proposition~\ref{prop:local} gives the local bound
$A\lesim_n2^{-aj}$. To pass to arbitrary radii, we apply
Lemma~\ref{lem:global} with \eqref{eq:kernel-C1} and choose
\begin{align}
    M=\lceil(a+3)j\rceil+3.
\end{align}
This gives
\begin{align}\label{eq:dyadic-L2}
    \|M_j f\|_2
    \lesim_n(1+j)^{n/2}2^{-aj}\|f\|_2.
\end{align}
The bound $|K_j|\lesim_n2^j\ind_{[-2,2]^n}$ also gives
\begin{align}\label{eq:rough-pointwise}
    M_j f\lesim_n2^j\Mrec f.
\end{align}
The $L^q$ boundedness of $\Mrec$ therefore implies
\begin{align}\label{eq:dyadic-Lq}
    \|M_j f\|_q\lesim_{n,q}2^j\|f\|_q,
    \qquad 1<q<\infty.
\end{align}
Finally, \eqref{eq:kernel-L1} gives
\begin{align}\label{eq:dyadic-Linfty}
    \|M_j f\|_\infty\lesim_n\|f\|_\infty.
\end{align}

Suppose first that
\begin{align}
    \frac{n+1}{n-1}<p<2.
\end{align}
We choose $1<q<p$, to be fixed below, and let $\theta\in(0,1)$
satisfy
\begin{align}\label{eq:interpolation-parameters}
    \frac1p=\frac{1-\theta}{q}+\frac{\theta}{2}.
\end{align}
Interpolating \eqref{eq:dyadic-Lq} with \eqref{eq:dyadic-L2} yields
\begin{align}\label{eq:interpolated}
    \|M_j f\|_p
    \lesim_{n,p,q}
    (1+j)^{n\theta/2}
    2^{j[1-(1+a)\theta]}\|f\|_p.
\end{align}
To justify the interpolation, fix a measurable choice of radii
and apply \eqref{eq:prelim-interpolation} to the resulting linear
operator. Since the bounds are uniform in this choice, the
linearization argument in Section~\ref{sec:preliminaries}
recovers the supremum.

It remains to choose $q$ so that the exponent of $2^j$ is negative.
As $q$ decreases to one, this exponent tends to
\begin{align}\label{eq:critical-exponent}
    1-\frac{n+1}{2}\left(1-\frac1p\right)
    =\frac{n+1}{2p}-\frac{n-1}{2}<0.
\end{align}
We can therefore choose $q>1$ sufficiently close to one that the
exponent is still negative. For this choice,
\begin{align}\label{eq:summable}
    \sum_{j\geq1}\|M_j\|_{L^p\to L^p}<\infty.
\end{align}

For $p=2$, \eqref{eq:summable} follows directly from
\eqref{eq:dyadic-L2}. For $2<p<\infty$, interpolation with
\eqref{eq:dyadic-Linfty} gives
\begin{align}
    \|M_j f\|_p
    \lesim_{n,p}(1+j)^{n/p}2^{-2aj/p}\|f\|_p,
\end{align}
so \eqref{eq:summable} holds in this case as well.

\subsection{Summation}

The smooth compactly supported kernel $K_0$ satisfies
\begin{align}
    M_0 f\lesim_n\Mrec f.
\end{align}
For each fixed $r>0$ and $f\in C_c(\R^n)$, dilation of
\eqref{eq:measure-decomposition} gives pointwise convergence of
the corresponding convolutions. Taking absolute values and then
the supremum over $r$, we obtain
\begin{align}\label{eq:pointwise-sum}
    \Mst f(x)\leq M_0 f(x)+\sum_{j\geq1}M_j f(x).
\end{align}
For finite $p$ in \eqref{eq:range}, we now apply the triangle
inequality and use \eqref{eq:summable}, together with the
$L^p$ boundedness of $\Mrec$. This proves \eqref{eq:main}.
For $p=\infty$, the bound is trivial.
\hfill$\square$

\normalem

\noindent Institut des Hautes Études Scientifiques, 91440 Bures-sur-Yvette, France \\
Email address:\\
chenmf@ihes.fr

\end{document}